\documentclass[12pt]{amsart}
\usepackage[headings]{fullpage}
\usepackage{amssymb,epic,eepic,epsfig,amsbsy,amsmath}
\usepackage[all]{xy}
\numberwithin{equation}{section}
                        \theoremstyle{plain}
\newcommand{\psdraw}[2]
         {\begin{array}{c} \hspace{-1.3mm}
         \raisebox{-4pt}{\psfig{figure=#1.eps,width=#2}}
         \hspace{-1.9mm}\end{array}}

\numberwithin{equation}{section}

\newtheorem{theorem}{Theorem}[section]
\newtheorem{lemma}[theorem]{Lemma}
\newtheorem{proposition}[theorem]{Proposition}

\newtheorem{thm}{Theorem}

\theoremstyle{definition}
\newtheorem{rmk}{Remark}

\def\BC{\mathbb C}

\def\BZ{\mathbb Z}

\DeclareMathOperator{\tr}{\mathrm tr}

\def\la{\langle}
\def\ra{\rangle}

\begin{document}

\title{The universal character ring of the $(-2,2m+1,2n)$-pretzel link}

\author[Anh T. Tran]{Anh T. Tran}
\address{Department of Mathematics, The Ohio State University, Columbus, OH 43210, USA}
\email{tran.350@osu.edu}

\thanks{2000 {\em Mathematics Classification:} 57M27.\\
{\em Key words and phrases: character variety, universal character ring, pretzel link.}}

\begin{abstract}
We explicitly calculate the universal character ring of the $(-2,2m+1,2n)$-pretzel link and show that it is reduced for all integers $m$ and $n$.
\end{abstract}

\maketitle

\setcounter{section}{-1}

\section{Introduction}

\subsection{The character variety and the universal character ring} The set of representations of a finitely presented group $G$ into $SL_2(\BC)$ is an algebraic set defined over $\BC$, on which
$SL_2(\BC)$ acts by conjugation. The set-theoretic
quotient of the representation space by that action does not
have good topological properties, because two representations with
the same character may belong to different orbits of that action. A better
quotient, the algebro-geometric quotient denoted by $X(G)$
(see \cite{LM}), has the structure of an algebraic
set. There is a bijection between $X(G)$ and the set of all
characters of representations of $G$ into $SL_2(\BC)$, hence
$X(G)$ is usually called the {\em character variety} of $G$. It is determined by the traces of some fixed elements $g_1, \cdots, g_k$ in $G$. More precisely, one can find $g_1, \cdots, g_k$ in $G$ such that for every element $g$ in $G$ there exists a polynomial $P_g$ in $k$ variables such that for any representation $\rho: G \to SL_2(\BC)$ one has $\tr(\rho(g)) = P_g(x_1, \cdots, x_k)$ where $x_j:=\tr(\rho(g_j))$. The {\em universal character ring} of $G$ is then defined to be the quotient of the polynomial ring $\BC[x_1, \cdots, x_k]$ by the ideal generated by all expressions of the form $\tr(\rho(u))-\tr(\rho(v))$, where $u$ and $v$ are any two words in the letters $g_1, \cdots, g_k$ which are equal in $G$, c.f. \cite{LTaj}. The universal character ring of $G$ is actually independent of the choice of $g_1, \cdots, g_k$. The quotient of the universal character ring of $G$ by its nil-radical is equal to the ring of regular functions on the character variety $X(G)$.

\subsection{Main results} Let $F_{a,w}:=\la a,w \ra$ be the free group in 2 letters $a$ and $w$. The character variety of $F_{a,w}$ is isomorphic to $\BC^3$ by the Fricke-Klein-Vogt theorem, see \cite{LM}. For every word $u$ in $F_{a,w}$ there is a \emph{unique} polynomial $P_u$ in 3 variables such that for any representation $\rho: F_{a,w} \to SL_2(\BC)$ one has $\tr (\rho(u))=P_u (x,y,z)$ where $x:=\tr(\rho(a)),~y:=\tr(\rho(w))$ and $z:=\tr(\rho(aw))$. For a word $u$ in $F_{a,w}$, we denote by $\overleftarrow{u}$ the word obtained from $u$ by writing the letters in $u$ in reversed order. In this paper we consider the group $$G:=\la a,w \mid r=\overleftarrow{r}\ra,$$ where $r$ is a word in $F_{a,w}$. For every representation $\rho: G \to SL_2(\BC)$, we consider $x,y,$ and $z$ as functions of $\rho$. The universal character ring of $G$ is calculated as follows.

\begin{thm}
The universal character ring of the group $\la a,w \mid r=\overleftarrow{r}\ra$ is the quotient of the polynomial ring $\BC[x,y,z]$ by the principal ideal generated by the polynomial $P_{raw}-P_{\overleftarrow{r}aw}.$
\label{main}
\end{thm}

In our joint work with T. Le on the AJ conjecture of \cite{Ga, Ge, FGL} which relates the A-polynomial and the colored Jones polynomials of a knot, it is important to know whether the universal character
ring of the knot group is reduced, i.e. whether its nilradical is zero \cite{Le06, LTaj}. So far
there are a few groups for which the universal character ring is known to be reduced:
free groups \cite{Si}, surface groups \cite{Sim1, Sim2, Si}, two-bridge knot groups \cite{Le93, PS}, torus knot groups
\cite{Mu, MO, Ma}, the $(-2,3,2n+1)$-pretzel knot groups \cite{LTaj}, and two-bridge link groups \cite{LTskein}.

In the present paper we consider the $(-2,2m+1,2n)$-pretzel link group, where $m$ and $n$ are integers. As an application of Theorem \ref{main} we will show that

\begin{thm}
(i) The fundamental group of the $(-2,2m+1,2n)$-pretzel link is isomorphic to the group $\la a,w \mid r=\overleftarrow{r}\ra$ where $r:=u^{n-1}awaw^{-1}a^{-1}$ and $u:=(awaw^{-1})^{1-m}w.$ Hence its universal character ring is the quotient of the polynomial ring $\BC[x,y,z]$ by the principal ideal generated by the polynomial $$
P_{raw}-P_{\overleftarrow{r}aw} = (xyz+4-x^2-y^2-z^2)[(xz-y)S_{n-1}(\alpha)-(S_{m}(\beta)-S_{m-1}(\beta))S_{n-2}(\alpha)],
$$
where 
\begin{eqnarray*}
\alpha &:=&P_u=yS_{m-1}(\beta)-(xz-y)S_{m-2}(\beta),\\
\beta &:=& P_{awaw^{-1}}=xyz+2-y^2-z^2, 
\end{eqnarray*}
and $S_k(\gamma)$ are the Chebyshev polynomials defined by $S_0(\gamma)=1,~S_1(\gamma)=\gamma$ and $S_{k+1}(\gamma)=\gamma S_k(\gamma)-S_{k-1}(\gamma)$ for all integer $k$. 

(ii) The universal character ring of the $(-2,2m+1,2n)$-pretzel link is reduced for all integers $m$ and $n$.
\label{pretzel}
\end{thm}

\begin{rmk}
The universal character ring of the $(-2,2m+1,2n+1)$-pretzel knot is calculated in a related paper \cite{Tr}. However, its reducedness is not proved.
\end{rmk}

\subsection{Acknowledgements} We would like to thank T. Le for helpful discussions. We would also like the referee for comments and suggestions.

\medskip

The rest of the paper is devoted to the proof of Theorems \ref{main} and \ref{pretzel}.

\section{Proof of Theorem \ref{main}}

\begin{proposition}
Let $G:=\la a,w \mid u=v\ra$, where $u$ and $v$ are two words in $F_{a,w}$. Then the universal character ring of $G$ is the quotient of the polynomial ring $\BC[x,y,z]$ by the ideal generated by the five polynomials $P_u-P_v,~P_{ua}-P_{va},~P_{uw}-P_{vw},~P_{uaw}-P_{vaw}$ and $P_{uwa}-P_{vwa}.$
\label{prop}
\end{proposition}

\begin{proof}
Let $I$ be the ideal in $\BC[x,y,z]$ generated by the five  polynomials $P_u-P_v,~P_{ua}-P_{va},~P_{uw}-P_{vw},~P_{uaw}-P_{vaw}$ and $P_{uwa}-P_{vwa}.$ We need to show that $P_{ug}-P_{vg} \in I$ for every $g \in G$. The proof will be based on the identity
\begin{equation}
P_{BAC}+P_{BA^{-1}C}=P_A P_{BC}
\label{Cayley}
\end{equation}
for all matrices $A,B,C$ in $SL_2(\BC)$, which follows from the identity $A+A^{-1}=P_A I_{2 \times 2}$ where $I_{2 \times 2}$ is the $2 \times 2$ identity matrix. 

Let $g_1:=a$ and $g_2:=w$. We first show that $P_{ug}-P_{vg} \in I$ whenever $g=g_{i_1}^{m_1} g_{i_2}^{m_2}$, where $i_1, i_2$ are distinct positive integers $\le 2$ and $m_1, m_2 \in \BZ.$ We use induction on the integer $\eta=k_1+k_2$ where $k_j$ is defined to be $-m_j$ if $m_j \le 0$ and $m_j-1$ if $m_j>0.$ If $\eta=0$ then all the $m_j$ are 0 or 1, so $g$ is equal to $1,a,w,aw$ or $wa$ and hence $P_{ug}-P_{vg} \in I$ by definition. If $\eta >0$ then $k_1>0$ or $k_2>0.$ If $k_1>0$ then $m_1 \not=0,1.$ If $m_1<0$ then by applying the identity \eqref{Cayley} we have 
\begin{eqnarray*}
P_{ug}-P_{vg} &=& (P_{g_{i_1}} P_{ug_{i_1}g}-P_{ug_{i_1}^2g})-(P_{g_{i_1}} P_{vg_{i_1}g}-P_{vg_{i_1}^2g})\\
&=& P_{g_{i_1}}(P_{ug_{i_1}g}-P_{vg_{i_1}g})-(P_{ug_{i_1}^2g}-P_{vg_{i_1}^2g})
\end{eqnarray*}
where $P_{ug_{i_1}g}-P_{vg_{i_1}g}$ and $P_{ug_{i_1}^2g}-P_{vg_{i_1}^2g}$ are in $I$ by the induction hypothesis, hence $P_{ug}-P_{vg} \in I.$ A similar reduction works if $m_1>1.$ The case $k_2>0$ is similar.

Now let $g \in G$ be arbitrary. We may write $g$ in the form $g_{i_1}^{m_1} \cdots g_{i_r}^{m_r}$ where $i_1, \cdots, i_r$ are integers that are not necessarily distinct. We will prove by induction on $r$ that $P_{ug}-P_{vg} \in I.$

By the case already proved we may assume that $i_1, \cdots, i_r$ are not all distinct. Suppose that $i_k=i_l$ for some $k<l.$ Let
$$b=g_{i_1}^{m_1} \cdots g_{i_k}^{m_k}, \quad c=g_{i_{k+1}}^{m_{k+1}} \cdots g_{i_l}^{m_l}, \quad d=g_{i_{l+1}}^{m_{l+1}} \cdots g_{i_r}^{m_r}.$$
Then $g=bcd$. By applying the identity \eqref{Cayley} we have
\begin{eqnarray*}
P_{ubcd}-P_{vbcd} &=& (P_{ubd}P_c-P_{ubc^{-1}d})-(P_{vbd}P_c-P_{vbc^{-1}d})\\
&=&(P_{ubd}-P_{vbd})P_c-(P_{ubc^{-1}d}-P_{vbc^{-1}d})
\end{eqnarray*}
But $P_{ubd}-P_{vbd}$ and $P_{ubc^{-1}d}-P_{vbc^{-1}d}$ are in $I$ by the induction hypothesis, and hence $P_{ubcd}-P_{vbcd}$ is also in $I.$
\end{proof}

\begin{rmk}
The proof of Proposition \ref{prop} is similar to that of \cite[Prop 1.4.1]{CS}.
\end{rmk}

\begin{proposition}
For every words $u,v$ in $F_{a,w}$ one has $P_{uv} = P_{\overleftarrow{u}\overleftarrow{v}}.$
\label{zero}
\end{proposition}

\begin{proof}
It is easy to see from the definition of the operator $\overleftarrow{\cdot}$ that $\overleftarrow{uv}=\overleftarrow{v}\overleftarrow{u}.$ By \cite[Lem 3.2.2]{Le93}, for every word $s$ in $F_{a,w}$ we have $P_s=P_{\overleftarrow{s}}.$ Hence $P_{uv}=P_{\overleftarrow{uv}}=P_{\overleftarrow{v}\overleftarrow{u}}.$ The proposition follows since $P_{\overleftarrow{v}\overleftarrow{u}}=P_{\overleftarrow{u}\overleftarrow{v}}.$
\end{proof}

\subsection{Proof of Theorem \ref{main}} From Proposition \ref{prop} it follows that the universal character ring of the group $G=\la a,w \mid r=\overleftarrow{r}\ra$ is the quotient of the polynomial ring $\BC[x,y,z]$ by the ideal generated by the five polynomials $P_r-P_{\overleftarrow{r}},~P_{ra}-P_{\overleftarrow{r}a},~P_{rw}-P_{\overleftarrow{r}w},~P_{raw}-P_{\overleftarrow{r}aw}$ and $P_{rwa}-P_{\overleftarrow{r}wa}$. By Proposition \ref{zero} we have 
\begin{eqnarray*}
P_r-P_{\overleftarrow{r}}&=&0,\\
P_{ra}-P_{\overleftarrow{r}a}&=&0,\\
P_{rw}-P_{\overleftarrow{r}w}&=&0,\\
P_{raw}-P_{\overleftarrow{r}aw} &=& P_{\overleftarrow{r}wa}-P_{rwa}.
\end{eqnarray*}
Hence the universal character ring of $G$ is the quotient of the polynomial ring $\BC[x,y,z]$ by the principal ideal generated by the polynomial $P_{raw}-P_{\overleftarrow{r}aw}.$

\section{Proof of Theorem \ref{pretzel}}

\subsection{Proof of part (i)} The fundamental group of the $(-2,2m+1,2n)$-pretzel link is
$$\pi:=\la a,b,c \mid bab^{-1}=(ac)^{-m}c(ac)^m,~a^{-1}ba=(cb)^nb(cb)^{-n}\ra
$$
where $a, b, c$ are meridians depicted in Figure 1. 

\begin{figure}[htpb]
$$ \psdraw{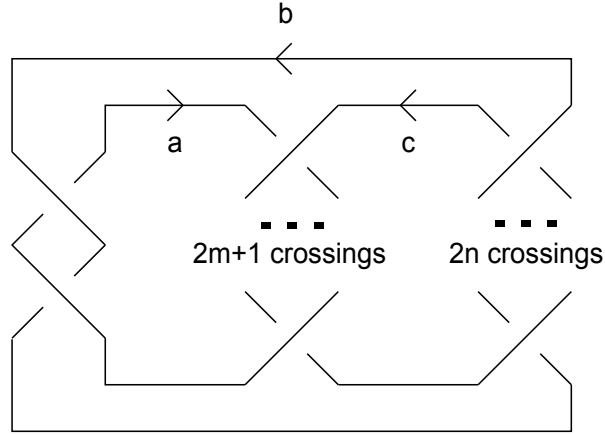}{3.1in} $$
\caption{The $(-2,2m+1,2n)$-pretzel link}
\end{figure}

The first relation in the group $\pi$ is $(ac)^m ba=c(ac)^m b$, i.e. $a(ca)^{m-1}cba=ca(ca)^{m-1}cb.$ Let $w=(ca)^{m-1}cb$ then  $awa=caw$. It implies that $ca=awaw^{-1}$ and $cb=(ca)^{1-m}w=(awaw^{-1})^{1-m}w.$ Let $u=(awaw^{-1})^{1-m}w$. Then $cb=u$ and so $$b=c^{-1}u=awa^{-1}w^{-1}a^{-1}(awaw^{-1})^{1-m}w=a(awaw^{-1})^{-m}w.$$
The second relation in the group $\pi$ becomes $(awaw^{-1})^{-m}wa=u^na(awaw^{-1})^{-m}wu^{-n}$, which is equivalent to $u^{n-1}awaw^{-1}a^{-1}=a^{-1}w^{-1}awau^{n-1}.$ Therefore
$$\pi=\la a, w \mid u^{n-1}awaw^{-1}a^{-1}=a^{-1}w^{-1}awau^{n-1} \ra.$$

\begin{lemma} One has $u=\overleftarrow{u}$, i.e. $u$ is a palindrome.
\label{u}
\end{lemma} 

\begin{proof} We first claim that $\overleftarrow{s^k}=\overleftarrow{s}^k$ for all integers $k$. Indeed, since $\overleftarrow{s}\overleftarrow{s^{-1}}=\overleftarrow{s^{-1}s}=1$ we obtain $\overleftarrow{s^{-1}}=\overleftarrow{s}^{-1}.$ If $k \ge 0$ then it is easy to prove by induction on $k$ that $\overleftarrow{s^k}=\overleftarrow{s}^k$. If $k<0$ then $\overleftarrow{s^k}=\overleftarrow{(s^{-1})^{-k}}=\left( \overleftarrow{s^{-1}} \right)^{-k}=(\overleftarrow{s}^{-1})^{-k}=\overleftarrow{s}^k.$

Applying the identity in the above claim with $s=awaw^{-1}$ and $k=1-m$ we get
$$\overleftarrow{u}=\overleftarrow{(awaw^{-1})^{1-m}w}=w(w^{-1}awa)^{1-m}=
w[w^{-1}(awaw^{-1})^{-m}awa]=(awaw^{-1})^{-m}awa.$$
It implies that $\overleftarrow{u}=(awaw^{-1})^{1-m}w=u.$
\end{proof}

Let $r:=u^{n-1}awaw^{-1}a^{-1}.$ Then, by Lemma \ref{u}, we have $\overleftarrow{r}=a^{-1}w^{-1}awa\overleftarrow{u}^{n-1}=a^{-1}w^{-1}awau^{n-1}.$ Hence $\pi=\la a,w \mid r=\overleftarrow{r} \ra$ and so, by Theorem \ref{main}, the universal character ring of $\pi$ is the quotient of the polynomial ring $\BC[x,y,z]$ by the principal ideal generated by the polynomial $P_{raw}-P_{\overleftarrow{r}aw}$, where $x=P_a,~y=P_w$ and $z=P_{aw}$. 

\begin{lemma}
Suppose the sequence $\{f_k\}_{k=-\infty}^{\infty}$ satisfies the recurrence relation $f_{k+1}=\gamma f_k-f_{k-1}$. Then $f_k=S_{k-1}(\gamma)f_1-S_{k-2}(\gamma)f_0,$ where $S_k(\gamma)$ are the Chebyshev polynomials defined by $S_0(\gamma)=1,~S_1(\gamma)=\gamma$ and $S_{k+1}(\gamma)=\gamma S_k(\gamma)-S_{k-1}(\gamma)$ for all integers $k$. 
\label{chev}
\end{lemma}

\begin{proof}
Let $\{g_k\}_{k=-\infty}^{\infty}$ be the sequence defined by $g_k=S_{k-1}(\gamma)f_1-S_{k-2}(\gamma)f_0.$ Then it is easy to see that $g_{k+1}=\gamma g_k-g_{k-1}$. Moreover, since $S_0(\gamma)=1$ and $S_{-1}(\gamma)=0$ we have $g_0=f_0,~g_1=f_1$. Therefore $g_k=f_k$.
\end{proof}

Let $\alpha=:P_u$ and $\beta:=P_{awaw^{-1}}.$

\begin{proposition}
One has \begin{eqnarray*}
\alpha &=&yS_{m-1}(\beta)-(xz-y)S_{m-2}(\beta),\\
\beta &=& xyz+2-y^2-z^2. 
\end{eqnarray*}
\label{alphabeta}
\end{proposition}

\begin{proof}
By applying the identity \eqref{Cayley} and Lemma \ref{chev} we have
\begin{eqnarray*}
\beta=P_{awaw^{-1}} &=& P_{awa}P_{w}-P_{awaw}\\
&=& (P_{aw}P_a-P_{awa^{-1}})P_w-(P_{aw}P_{aw}-P_{I_2})\\
&=& (zx-y)y-(z^2-2),\\
\alpha = P_u &=& P_{(awaw^{-1})^{-m}awa}\\
&=&P_{(awa)^{-1}(awaw^{-1})^m}\\
&=&P_{w^{-1}}S_{m-1}(\beta)-P_{(awa)^{-1}}S_{m-2}(\beta)\\
&=& yS_{m-1}(\beta)-(xz-y)S_{m-2}(\beta).
\end{eqnarray*}
This proves the proposition.
\end{proof}

\begin{proposition}
One has $$
P_{raw}-P_{\overleftarrow{r}aw} = (xyz+4-x^2-y^2-z^2)[(xz-y)S_{n-1}(\alpha)-(S_{m}(\beta)-S_{m-1}(\beta))S_{n-2}(\alpha)].
\label{P}
$$
\end{proposition}
\begin{proof} By applying the identity \eqref{Cayley} and Lemma \ref{chev} we have
\begin{eqnarray*}
P_{raw} &=& P_{u^{n-1}awa}=P_{u^{n}w^{-1}(awaw^{-1})^mw} \\
&=& P_{awa}S_{n-1}(\alpha)-P_{w^{-1}(awaw^{-1})^mw}S_{n-2}(\alpha)\\
&=& (xz-y)S_{n-1}(\alpha)-(\beta S_{m-1}(\beta)-2S_{m-2}(\beta))S_{n-2}(\alpha),\\
P_{\overleftarrow{r}aw} &=& P_{a^{-1}w^{-1}awau^{n-1}aw}= P_{a^{-1}w^{-1}(awaw^{-1})^m u^{n}aw}\\
&=& P_{a^{-1}w^{-1}awaaw}S_{n-1}(\alpha)-P_{a^{-1}w^{-1}(awaw^{-1})^m aw}S_{n-2}(\alpha)\\
&=& P_{a^{-1}w^{-1}awaaw}S_{n-1}(\alpha)-(P_{a^{-1}w^{-1}awaw^{-1} aw}S_{m-1}(\beta)-P_{a^{-1}w^{-1}aw}S_{m-2}(\beta))S_{n-2}(\alpha)
\end{eqnarray*}
where 
\begin{eqnarray*}
P_{a^{-1}w^{-1}awaaw}  &=& P_{awa}P_{a^{-1}w^{-1}aw}-P_{a^{-1}w^{-1}(awa)^{-1}aw}\\
&=&P_{awa}(P_aP_{w^{-1}aw}-P_{aw^{-1}aw})-P_{a^{-1}w^{-1}a^{-1}}\\
&=&(xz-y)(x^2-\beta-1),\\
P_{a^{-1}w^{-1}awaw^{-1} aw} &=& P_{awaw^{-1}}P_{a^{-1}w^{-1}aw}-P_{a^{-1}w^{-1}(awaw^{-1})^{-1}aw}\\
&=&P_{awaw^{-1}}(P_aP_{w^{-1}aw}-P_{aw^{-1}aw})-P_{a^{-2}}\\
&=&\beta(x^2-\beta)-(x^2-2).
\end{eqnarray*}
Hence 
\begin{eqnarray*}
P_{\overleftarrow{r}aw} &=& (xz-y)(x^2-\beta-1)S_{n-1}(\alpha)\\
&&-\,\left((\beta(x^2-\beta)-(x^2-2))S_{m-1}(\beta)-(x^2-\beta)S_{m-2}(\beta)\right)S_{n-2}(\alpha),
\end{eqnarray*}
and so
\begin{eqnarray*}
P_{raw}-P_{\overleftarrow{r}aw} &=&(\beta+2-x^2)[(xz-y)S_{n-1}(\alpha)-((\beta -1)S_{m-1}(\beta)-S_{m-2}(\beta))S_{n-2}(\alpha)]\\
&=& (\beta+2-x^2)[(xz-y)S_{n-1}(\alpha)-(S_{m}(\beta)-S_{m-1}(\beta))S_{n-2}(\alpha)].
\end{eqnarray*}
This proves the proposition since $\beta+2-x^2=xyz+4-x^2-y^2-z^2$.
\end{proof}

Part (i) of Theorem \ref{pretzel} follows from Propositions \ref{alphabeta} and \ref{P}.

\subsection{Proof of part (ii)} Recall from Proposition \ref{alphabeta} that $\alpha = yS_{m-1}(\beta)-(xz-y)S_{m-2}(\beta)$ and $\beta = xyz+2-y^2-z^2$. Let $$Q(x,y,z)=(xz-y)S_{n-1}(\alpha)-(S_{m}(\beta)-S_{m-1}(\beta))S_{n-2}(\alpha).$$ Then, by Proposition \ref{P}, $P_{raw}-P_{\overleftarrow{r}aw} = (xyz+4-x^2-y^2-z^2)Q(x,y,z).$

\begin{proposition}
One has
$$Q(x,y,0)=(-1)^{(m-1)(n-1)} \, S_{2mn-2m-n-2}(y).$$
\label{z=0}
\end{proposition}

\begin{proof}
Fix $z=0$. Then we have $\beta=2-y^2,~\alpha=y(S_{m-1}(\beta)+S_{m-2}(\beta))$ and 
$$Q=-[yS_{n-1}(\alpha)+(S_{m}(\beta)-S_{m-1}(\beta))S_{n-2}(\alpha)].$$
Let $y=a+a^{-1}.$ Then $\beta=-a^2-a^{-2}$ and so
\begin{eqnarray*}
\alpha &=& y(S_{m-1}(\beta)+S_{m-2}(\beta))\\
&=& (a+a^{-1})\left(\frac{(-a^2)^m-(-a^{-2})^m}{(-a^2)-(-a^{-2})}+\frac{(-a^2)^{m-1}-(-a^{-2})^{m-1}}{(-a^2)-(-a^{-2})}\right)\\
&=& (-1)^{m-1}(a^{2m-1}+a^{1-2m}).
\end{eqnarray*}
Hence 
\begin{eqnarray*}
-Q &=& yS_{n-1}(\alpha)+(S_{m}(\beta)-S_{m-1}(\beta))S_{n-2}(\alpha)\\
&=& (a+a^{-1})\frac{((-1)^{m-1}a^{2m-1})^n-((-1)^{m-1}a^{1-2m})^n}{(-1)^{m-1}a^{2m-1}-(-1)^{m-1}a^{1-2m}}\\
&&+\,\left(\frac{(-a^2)^{m+1}-(-a^{-2})^{m+1}}{(-a^2)-(-a^{-2})}-\frac{(-a^2)^{m}-(-a^{-2})^{m}}{(-a^2)-(-a^{-2})}\right)\\
&&\times \, \frac{((-1)^{m-1}a^{2m-1})^{n-1}-((-1)^{m-1}a^{1-2m})^{n-1}}{(-1)^{m-1}a^{2m-1}-(-1)^{m-1}a^{1-2m}}\\
&=&(-1)^{(m-1)(n-1)}(a+a^{-1})\frac{a^{(2m-1)n}-a^{(1-2m)n}}{a^{2m-1}-a^{1-2m}}\\
&&+\,(-1)^{m+(m-1)(n-2)}\frac{a^{2m+1}-a^{-(2m+1)}}{a-a^{-1}} \times \frac{a^{(2m-1)(n-1)}-a^{(1-2m)(n-1)}}{a^{2m-1}-a^{1-2m}}\\
&=&(-1)^{(m-1)(n-1)}\frac{a^{-2mn+2m+n+1}-a^{2mn-2m-n-1}}{a-a^{-1}}\\
&=&(-1)^{(m-1)(n-1)+1} \, S_{2mn-2m-n-2}(y).
\end{eqnarray*}
The proposition follows.
\end{proof}

For two polynomials $f,g$ in $\BC[x,y,z]$, we say that they are \emph{$y$-equal}, and write
$$f =_y g$$
if their $y$-degrees are equal and the coefficients of their highest powers in $y$ are also equal. 

\begin{proposition}
One has
$$Q(x,y,z) =_y \begin{cases}  (-1)^{(m-1)(n-1)} z^2y^{2mn-2m-n} \quad & \emph{if} \quad m \ge 2 \emph{~and~} n \ge 2,\\
-(-1)^{(m-1)(n-1)} y^{-2mn+2m+n} \quad & \emph{if} \quad m \ge 2 \emph{~and~} n \le 1,\\
z^2 y^{n-2} \quad & \emph{if} \quad m=1 \emph{~and~} n \ge 3,\\
z^2-1 & \emph{if} \quad m=1 \emph{~and~} n=2,\\
-y^{2-n} & \emph{if} \quad m=1 \emph{~and~} n \le 1,\\
(-1)^n y^{n} & \emph{if} \quad m=0 \emph{~and~} n \ge 0,\\
0 & \emph{if} \quad m=0 \emph{~and~} n=-1,\\
(-1)^{n-1}y^{-(n+2)} & \emph{if} \quad m=0 \emph{~and~} n \le -2,\\
-(-1)^{(m-1)(n-1)}y^{-2mn+2m+n} \quad & \emph{if} \quad m \le -1 \emph{~and~} n \ge 1,\\
(-1)^{(m-1)(n-1)} y^{2mn-2m-n-2} \quad & \emph{if} \quad m \le -1 \emph{~and~} n \le 0.
             \end{cases}$$
             \label{degree}
\end{proposition}

\begin{proof} We first prove the following result

\begin{lemma}
One has 

(i) $\alpha =_y (-1)^{m-1}y^{|2m-1|}$.

(ii) $S_{m}(\beta)-S_{m-1}(\beta) =_y \begin{cases}  (-1)^{m}y^{2m} \quad & \emph{if} \quad m \ge 0,\\
(-1)^{m-1}y^{-2(m+1)} & \emph{if} \quad m \le -1.
             \end{cases} $
\label{alpha}
\end{lemma}

\begin{proof}

(i) Note that $\beta =_y -y^2$. If $m \ge 2$ then 
$$S_{m-1}(\beta)=_y S_{m-1}(-y^2)=_y(-y^2)^{m-1}=(-1)^{m-1}y^{2m-2}.$$
Similarly $S_{m-2}(\beta)=_y (-1)^{m-2}y^{2m-4}.$ Hence
$$\alpha = yS_{m-1}(\beta)-(xz-y)S_{m-2}(\beta)=_y (-1)^{m-1}y^{2m-1}.$$
If $m=1$ then $\alpha=y$. If $m=0$ then $\alpha = xz-y$. If $m \le -1$ then let $m'=-(m+1) \ge 0$. Note that $S_k(\gamma)=-S_{-k-2}(\gamma)$ for all integers $k$. Hence
$$S_{m-1}(\beta) = -S_{-m-1}(\beta) = -S_{m'}(\beta) =_y -S_{m'}(-y^2)=-(-1)^{m'}y^{2m'}=(-1)^{m}y^{-2(m+1)}.$$
Similarly $S_{m-2}(\beta)=-S_{m'+1}(\beta) =_y (-1)^{m-1}y^{-2m}$.
Hence
$$\alpha = yS_{m-1}(\beta)-(xz-y)S_{m-2}(\beta)=_y (-1)^{m-1}y^{1-2m}.$$

(ii) Similar to (i).
\end{proof}

\subsubsection{The case $m=0$} Then $\alpha = xz-y$ and so
\begin{eqnarray*}
Q &=& (xz-y)S_{n-1}(xz-y)-S_{n-2}(xz-y)\\
&=& S_{n}(xz-y) =_y 
\begin{cases}  (-1)^n y^{n} \quad & \text{if} \quad n \ge 0,\\
0 & \text{if} \quad n=-1,\\
(-1)^{n-1}y^{-(n+2)} & \text{if} \quad n \le -2.
             \end{cases}
             \end{eqnarray*}

\subsubsection{The case $m \le -1$} Then, by Lemma \ref{alpha}, $\alpha =_y (-1)^{m-1}y^{1-2m}$ and $S_{m}(\beta)-S_{m-1}(\beta) =_y (-1)^{m-1}y^{-2(m+1)}$. If $n \ge 2$ then 
\begin{eqnarray*}
S_{n-1}(\alpha) &=_y& ((-1)^{m-1}y^{1-2m})^{n-1}=(-1)^{(m-1)(n-1)} y^{-2mn+2m+n-1},\\
S_{n-2}(\alpha) &=_y& ((-1)^{m-1}y^{1-2m})^{n-2}=(-1)^{(m-1)(n-2)} y^{-2mn+4m+n-2}.
\end{eqnarray*}
It follows that
\begin{eqnarray*}
(xz-y)S_{n-1}(\alpha) &=_y& -(-1)^{(m-1)(n-1)} y^{-2mn+2m+n},\\
(S_{m}(\beta)-S_{m-1}(\beta))S_{n-2}(\alpha) &=_y& (-1)^{m-1}y^{-2(m+1)}(-1)^{(m-1)(n-2)} y^{-2mn+4m+n-2}\\
&=& (-1)^{(m-1)(n-1)} y^{-2mn+2m+n-4}.
\end{eqnarray*}
Hence $Q=(xz-y)S_{n-1}(\alpha)-(S_{m}(\beta)-S_{m-1}(\beta))S_{n-2}(\alpha) =_y -(-1)^{(m-1)(n-1)} y^{-2mn+2m+n}.$

If $n=1$ then $Q=xz-y$. If $n=0$ then $Q=S_{m}(\beta)-S_{m-1}(\beta)=_y(-1)^{m-1}y^{-2(m+1)}.$ Similarly, if $n \le -1$ then $Q=_y(-1)^{(m-1)(n-1)} y^{2mn-2m-n-2}$. Hence 
$$Q =_y \begin{cases}  -(-1)^{(m-1)(n-1)}y^{-2mn+2m+n} \quad & \text{if} \quad m \le -1 \text{~and~} n \ge 1,\\
(-1)^{(m-1)(n-1)} y^{2mn-2m-n-2} \quad & \text{if} \quad m \le -1 \text{~and~} n \le 0.
             \end{cases}$$

If $m \ge 1$ then we write 
\begin{eqnarray*}
Q &=&  (xz-y)S_{n-1}(\alpha)-((\beta -1) S_{m-1}(\beta)-S_{m-2}(\beta))S_{n-2}(\alpha)\\
&=& (xz-y)S_{n-1}(\alpha)-((xyz+1-y^2-z^2)S_{m-1}(\beta)-S_{m-2}(\beta))S_{n-2}(\alpha)\\
&=& (xz-y)(S_{n-1}(\alpha)-yS_{m-1}(\beta)S_{n-2}(\alpha))+((z^2-1)S_{m-1}(\beta)+S_{m-2}(\beta))S_{n-2}(\alpha)\\
&=& (xz-y)(S_{n-1}(\alpha)-(\alpha+(xz-y)S_{m-2}(\beta))S_{n-2}(\alpha))\\
&& + \, ((z^2-1)S_{m-1}(\beta)+S_{m-2}(\beta))S_{n-2}(\alpha)\\
&=&-(xz-y)(S_{n-3}(\alpha)+(xz-y)S_{m-2}(\beta)S_{n-2}(\alpha))\\
&& + \, ((z^2-1)S_{m-1}(\beta)+S_{m-2}(\beta))S_{n-2}(\alpha)\\
&=&[(z^2-1)S_{m-1}(\beta)-((xz-y)^2-1)S_{m-2}(\beta)]S_{n-2}(\alpha)-(xz-y)S_{n-3}(\alpha)\\
&=&[(z^2-1)S_{m-1}(\beta)-(-xyz+x^2z^2-z^2+1-\beta)S_{m-2}(\beta)]S_{n-2}(\alpha)-(xz-y)S_{n-3}(\alpha)\\
&=&[z^2 S_{m-1}(\beta)+(xyz-x^2z^2+z^2-1)S_{m-2}(\beta)+S_{m-3}(\beta)]S_{n-2}(\alpha)-(xz-y)S_{n-3}(\alpha)
\end{eqnarray*}
Let $\delta=z^2 S_{m-1}(\beta)+(xyz-x^2z^2+z^2-1)S_{m-2}(\beta)+S_{m-3}(\beta).$ Then
$$Q=\delta S_{n-2}(\alpha)-(xz-y)S_{n-3}(\alpha).$$

\begin{lemma} One has
$$\delta =_y \begin{cases}  (-1)^{m-1}z^2y^{2m-2} \quad & \emph{if} \quad m \ge 2,\\
z^2-1 & \emph{if} \quad m=1,\\
(-1)^{m}y^{2-2m} & \emph{if} \quad m \le 0.
             \end{cases} $$
             \label{gamma}
\end{lemma}

\subsubsection{The case $m=1$} In this case $\alpha=y$ and so $$Q=(z^2-1)S_{n-2}(y)-(xz-y)S_{n-3}(y)=z^2S_{n-2}(y)-xzS_{n-3}(y)+S_{n-4}(y).$$
Hence
$$Q =_y \begin{cases}  z^2y^{n-2} \quad & \text{if} \quad n \ge 3,\\
z^2-1 & \text{if} \quad n=2,\\
-y^{2-n} & \text{if} \quad n \le 1.
             \end{cases} $$
             
\subsubsection{The case $m \ge 2$} Then, by Lemmas \ref{alpha} and \ref{gamma}$, \alpha =_y (-1)^{m-1}y^{2m-1}$ and $\delta =_y (-1)^{m-1}z^2y^{2m-2}$. By similar arguments as in the case $m \le -1$, we obtain 
$$Q =_y \begin{cases}  (-1)^{(m-1)(n-1)} z^2y^{2mn-2m-n} \quad & \text{if} \quad m \ge 2 \text{~and~} n \ge 2,\\
-(-1)^{(m-1)(n-1)} y^{-2mn+2m+n} \quad & \text{if} \quad m \ge 2 \text{~and~} n \le 1.
             \end{cases}$$
This completes the proof of Proposition \ref{degree}.
\end{proof}

From Propositions \ref{z=0} and \ref{degree}, we have

(i) If $Q(x,y,z)$ has non-trivial repeated factors then so has $Q(0,y,z)$. Moreover, if $R(y,z)$ is a non-trivial repeated factor of $Q(0,y,z)$ then the coefficient of the highest power of $y$ in $R(y,z)$ is a divisor of $z$.

(ii) The difference of the $y$-degrees of $Q(0,y,z)$ and $Q(0,y,0)=\pm S_{2mn-2m-n-2}(y)$ is at most 2.

Let us now prove part (ii) of Theorem \ref{pretzel}. The goal is to show that $$P_{raw}-P_{\overleftarrow{r}aw} = (xyz+4-x^2-y^2-z^2)Q(x,y,z)$$ does not have any non-trivial repeated factors.

Suppose that $Q(x,y,z)$ has non-trivial repeated factors. Then $Q(0,y,z)$ also has non-trivial repeated factors. Let $R(y,z)$ be a non-trivial repeated factor of $Q(0,y,z)$. Note that the coefficient of the highest power of $y$ in $R(y,z)$ is a divisor of $z$. If $R$ has $y$-degree 0 then $R=\pm z$. It implies that $z$ is a divisor of $Q(0,y,z)$ and so $Q(0,y,0)=\pm S_{2mn-2m-n-2}(y)=0.$ Hence $2mn-2m-n-2=-1$, i.e. ($m=0$ and $n=-1$) or ($m=1$ and $n=3$). If $m=0$ and $n=-1$ then $\alpha=xz-y$ and so $Q(x,y,z)=0$. If $m=1$ and $n=3$ then $\alpha=y$ and so $Q(x,y,z) = z(zy-x)$ does not have any non-trivial repeated factors. 

We consider the case that $R$ has $y$-degree $k \ge 1$. Let $r_k$ be the coefficient of $y^k$ in $R(y,z)$. If $r_k= \pm 1$ then $R(y,0)$ is a non-trivial repeated factor of $Q(0,y,0)=\pm S_{2mn-2m-n-2}(y)$. This is impossible since $S_{2mn-2m-n-2}(y)$ does not have any non-trivial repeated factors. Hence $r_k= \varepsilon z$, where $\varepsilon=\pm 1$, and so $R(y,z)=\varepsilon zy^k+r_{k-1}y^{k-1}+\cdots.$ Since the difference of the $y$-degrees of $Q(0,y,z)$ and $Q(0,y,0)$ is at most 2, the $y$-degree of $R(y,0)$ is exactly $k-1$. If $k \ge 2$ then $R(y,0)$ is a non-trivial repeated factor of $Q(0,y,0)=\pm S_{2mn-2m-n-2}(y)$, which is impossible. Hence $k=1$ and so $R(y,z)=\varepsilon zy+r_{0}(z)$ where $r_0(0) \not= 0.$ We have 
$$Q(0,y,z)=-[yS_{n-1}(\alpha \mid_{x=0})+(S_{m}(\beta\mid_{x=0})-S_{m-1}(\beta\mid_{x=0}))S_{n-2}(\alpha\mid_{x=0})]$$ 
where $\beta\mid_{x=0} \, =2-y^2-z^2$ and $\alpha\mid_{x=0} \,= y[S_{m-1}(\beta\mid_{x=0})+S_{m-2}(\beta\mid_{x=0})].$ It implies that $Q(0,y,z)$ contains even powers of $z$ only. Since $R(y,z)=\varepsilon z+r_0(z)$ is a non-trivial repeated factor of $Q(0,y,z)$, so is $R(y,-z)=\varepsilon (-z)+r_0(-z)$. If $R(y,-z) \not= -R(y,z)$, then $R(y,z)$ and $R(y,-z)$ are distinct non-trivial repeated factors in the prime factorization of $Q(0,y,z)$ in the UFD $\BC[y,z]$. It implies that the difference of the $y$-degrees of $Q(0,y,z)$ and $Q(0,y,0)$ is at least 4, a contradiction. Hence $R(y,-z)=-R(y,z)$, which means that $r_0(-z)=-r_{0}(z)$, i.e. $r_0(z)$ is an odd polynomial in $z$. This contradicts the condition that $r_0(0) \not= 0.$ Therefore $Q(x,y,z)$ does not have any non-trivial repeated factors. 

It remains to show that $xyz+4-x^2-y^2-z^2$ is not a divisor of $Q(x,y,z)$ unless $Q(x,y,z) \equiv 0$. From Proposition \ref{degree}, it is easy to see that $Q(x,y,z) \equiv 0$ if and only if $m=0$ and $n=-1$. Suppose $Q(x,y,z) \not\equiv 0$. If $m=1$ and $n=3$ then $Q(x,y,z)=z(zy-x)$. Otherwise $Q(x,y,0)=\pm S_{2mn-2m-n-2}(y)\not\equiv 0$ is not divisible by $4-x^2-y^2$. It implies that $Q(x,y,z)$ is not divisible by $xyz+4-x^2-y^2-z^2$. Therefore $P_{raw}-P_{\overleftarrow{r}aw} = (xyz+4-x^2-y^2-z^2)Q(x,y,z)$ does not have any non-trivial repeated factors and so the universal character ring of the $(-2,2m+1,2n)$-pretzel link is reduced for all integers $m$ and $n$.

\end{document}